\documentclass[12pt]{amsart}
\usepackage{amsmath, amsthm, amscd, amsfonts}

\newtheorem{theorem}{Theorem}[section]

\newtheorem{proposition}[theorem]{Proposition}
\newtheorem{corollary}[theorem]{Corollary}
\theoremstyle{definition}

\theoremstyle{remark}

\numberwithin{equation}{section}

\newfont{\kh}{msbm10}

\begin{document}
\title[Normality of adjointable module maps]
{Normality of adjointable module maps}
\author{K. Sharifi}
\address{Kamran Sharifi, \newline Department of Mathematics,
Shahrood University of Technology, P. O. Box 3619995161-316,
Shahrood, Iran} \email{sharifi.kamran@gmail.com and
sharifi@shahroodut.ac.ir}


\begin{abstract}Normality of bounded and unbounded
adjointable operators are discussed. Suppose $T$ is an adjointable
operator between Hilbert C*-modules which has polar
decomposition, then $T$ is normal if and only if there exists a
unitary operator $ \mathcal{U}$ which commutes with $T$ and $T^*$
such that $T=\mathcal{U} \, T^*.$ Kaplansky's theorem for
normality of the product of bounded operators is also reformulated
in the framework of Hilbert C*-modules.
\end{abstract}
\maketitle

\section{Introduction and preliminary.}
Normal operators may be regarded as a generalization of a
selfadjoint operator $T$ in which $T^*$ need not be exactly $T$
but commutes with $T$. They form an attractive and important class
of operators which play a vital role in operator theory,
especially, in spectral theory. In this note we will study bounded
and unbounded normal module maps between Hilbert C*-modules which
have polar decomposition. Indeed, for adjointable operator $T$
between Hilbert C*-modules which has polar decomposition, we
demonstrate that $T$ is normal if and only if there exists a
unitary operator $ \mathcal{U}$ such that $T=\mathcal{U} \, T^*$.
In this situation, $\mathcal{U} \,T \subseteq T \, \mathcal{U}$
and $\mathcal{U} \,T^* \subseteq T^* \, \mathcal{U}$ (compare
\cite[Problem 13, page 109]{Horn/Matrix}). The results are
interesting even in the case of Hilbert spaces.

Suppose $T, S$ are bounded adjointable operators between Hilbert
C*-modules. Suppose $T$ has polar decomposition and $T$ and $TS$
are normal operators. Then we show that $ST$ is a normal operator
if and only if $S$ commutes with $|T|$. This fact has been proved
by Kaplansky \cite{Kaplansky/normality1953} in the case of Hilbert
spaces.

Throughout the present paper we assume $\mathcal{A}$ to be an
arbitrary C*-algebra. We deal with bounded and unbounded
operators at the same time, so we denote bounded operators by
capital letters and unbounded operators by small letters. We use
the notations $Dom(.)$, $Ker(.)$ and $Ran(.)$ for domain, kernel
and range of operators, respectively.

Hilbert C*-modules are essentially objects like Hilbert spaces,
except that the inner product, instead of being complex-valued,
takes its values in a C*-algebra. Although Hilbert C*-modules
behave like Hilbert spaces in some ways, some fundamental Hilbert
space properties like Pythagoras' equality, self-duality, and even
decomposition into orthogonal complements do not hold. A (right)
{\it pre-Hilbert C*-module} over a C*-algebra $\mathcal{A}$ is a
right $\mathcal{A}$-module $X$ equipped with an
$\mathcal{A}$-valued inner product $\langle \cdot , \cdot \rangle
: X \times X \to \mathcal{A}\,, \ (x,y) \mapsto \langle x,y
\rangle$, which is $\mathcal A$-linear in the second variable $y$
and has the properties:
$$ \langle x,y \rangle=\langle y,x \rangle ^{*}, \
\  \langle x,x \rangle \geq 0 \ \ {\rm with} \
   {\rm equality} \ {\rm only} \ {\rm when} \  x=0.$$
A pre-Hilbert $\mathcal{A}$-module $X$ is called a {\it Hilbert $
\mathcal{A}$-module} if $X$ is a Banach space with respect to the
norm $\| x \|=\|\langle x,x\rangle \| ^{1/2}$. A Hilbert
$\mathcal A$-submodule $W$ of a Hilbert $\mathcal A$-module $X$ is
an orthogonal summand if $W \oplus W^\bot = X$, where $W^\bot$
denotes the orthogonal complement of $W$ in $X$. We denote by
$\mathcal{L}(X)$ the C*-algebra of all adjointable operators on
$X$, i.e., all $\mathcal A$-linear maps $T :X \rightarrow X$ such
that there exists $T^*:X \rightarrow X$ with the property $
\langle Tx,y \rangle =\langle x,T^*y \rangle$ for all $ x,y \in
X$. A bounded adjointable operator $ \mathcal{V} \in
\mathcal{L}(X)$ is called a {\it partial isometry} if $
\mathcal{V} \, \mathcal{V}^*  \mathcal{V} = \mathcal{V}$, see
\cite{SHA/PARTIAL} for some equivalent conditions. For the basic
theory of Hilbert C*-modules we refer to the books \cite{LAN,
WEG} and the papers \cite{B-G/class, FR2}.

An unbounded regular operator on a Hilbert C*-module is an
analogue of a closed operator on a Hilbert space. Let us quickly
recall the definition. A densely defined closed
$\mathcal{A}$-linear map $t: Dom(t) \subseteq X \to X$ is called
{\it regular} if it is adjointable and the operator $1+t^*t$ has
a dense range.  Indeed, a densely defined operator $t$ with a
densely defined adjoint operator $t^*$ is regular if and only if
its graph is orthogonally complemented in $X \oplus X$ (see e.g.
\cite{F-S, LAN}). We denote the set of all regular operators on
$X$ by $\mathcal{R}(X)$. If $t$ is regular then $t^*$ is regular
and $t=t ^{**}$, moreover $t^*t$ is regular and selfadjoint.
Define $Q_{t}=(1+t^*t)^{-1/2}$ and $F_{t}=tQ_{t}$, then
$Ran(Q_{t})=Dom(t)$,  $0 \leq Q_{t}= (1 - F_{t}^{*} F_{t})^{1/2}
\leq 1$ in $\mathcal{L}(X)$ and $F_{t}\in \mathcal{L}(X)$
\cite[(10.4)]{LAN}. The bounded operator $F_{t}$ is called the
bounded transform of regular operator $t$. According to
\cite[Theorem 10.4]{LAN}, the map $t\to F_{t}$ defines an
adjoint-preserving bijection
$$\mathcal{R}(X) \to \{ F \in \mathcal{L}(X) :\| F \|\leq 1 \ \, {\rm and} \ \
Ran(1- F^* F ) \ {\rm is} \ {\rm dense \ in} \ X \}.
$$

Very often there are interesting relationships between regular
operators and their bounded transforms. In fact, for a regular
operator $t$, some properties transfer to its bounded transform
$F_{t}$, and vice versa. Suppose $t \in \mathcal{R}(X)$ is a
regular operator, then  $t$ is called {\it normal} iff
$Dom(t)=Dom(t^*)$ and $\langle tx,tx \rangle=\langle t^*x,t^*x
\rangle $ for all $x \in Dom(t)$. $t$ is called {\it selfadjoint}
iff $t^*=t$ and $t$ is called {\it positive} iff $t$ is normal and
$\langle tx,x \rangle \geq 0$ for all $ x \in Dom(t)$. In
particular, a regular operator $t$ is normal (resp., selfadjoint,
positive) iff its bounded transform $F_t$ is normal (resp.,
selfadjoint, positive). Moreover, both $t$ and $F_t$ have the
same range and the same kernel. If $t \in \mathcal{R}(X)$ then
$Ker(t)=Ker(|t|)$ and $\overline{Ran(t^*)}= \overline{Ran(|t|)}$
, cf. \cite{KUS}. If $t  \in \mathcal{R}(X)$ is a normal operator
then $Ker(t)=Ker(t^*)$ and $\overline{Ran(t)}=
\overline{Ran(t^*)}$.

A bounded adjointable operator $T$ has polar decomposition if and
only if $\overline{Ran(T)}$ and $\overline{Ran(|T|)}$ are
orthogonal direct summands \cite[Theorem 15.3.7]{WEG}. The result
has been generalized in Theorem 3.1 of \cite{FS2} for regular
operators. Indeed, for $t \in \mathcal{R}(X)$ the following
conditions are equivalent:
\begin{itemize}
\item $t$ has a unique polar decomposition $t= \mathcal{V}|t|$,
where $\mathcal{V} \in \mathcal{L}(X)$ is a partial isometry for
which $Ker(\mathcal{V})=Ker(t)$.

\item $X=Ker(|t|) \oplus \overline{Ran(|t|)}$ and $X=Ker(t^*) \oplus
\overline{Ran(t)}$.

\item The adjoint operator $t^*$ has polar decomposition $t^*=
\mathcal{V^*}|t^*|$.

\item The bounded transform $F_t$ has polar decomposition
$F_{t}= \mathcal{V}|F_{t}|$.
\end{itemize}
\noindent In this situation, $\mathcal{V^*}\mathcal{V}|t|=|t|$,
$\mathcal{V^*}t=|t|$ and $\mathcal{V}\mathcal{V^*}t=t$, moreover,
we have  $Ker(\mathcal{V^*})=Ker(t^*)$,
$Ran(\mathcal{V})=\overline{Ran(t)}$ and $Ran(\mathcal{V^*})=
\overline{Ran(t^*)}$. That is, $ \mathcal{V} \mathcal{V^*}$ and $
\mathcal{V^*} \mathcal{V}$ are orthogonal projections onto the
submodules $ \overline{Ran(t)}$ and $ \overline{Ran(t^*)}$,
respectively.

The above facts and Proposition 1.2 of \cite{F-S} show that each
regular operator with closed range has polar decomposition.

Recall that an arbitrary C*-algebra of compact operators
$\mathcal{A}$ is a $c_{0}$-direct sum of elementary C*-algebras
$\mathcal{K}(H_{i})$ of all compact operators acting on Hilbert
spaces $H_{i}, \ i \in I$,  cf.~\cite[Theorem 1.4.5]{ARV}.
Generic properties of Hilbert C*-modules over C*-algebras of
compact operators have been studied systematically in \cite{ARA,
B-G, F-S, FS2, GUL} and references therein. If $\mathcal{A}$ is a
C*-algebra of compact operators then for every Hilbert $
\mathcal{A}$-module $X$, every densely defined closed operator
$t: Dom(t)\subseteq X \to X$ is automatically regular and has
polar decomposition, cf. \cite{F-S, FS2, GUL}.

The stated results also hold for bounded adjointable operators,
since $ \mathcal{L}(X)$ is a subset of $\mathcal{R}(X)$. The space
$\mathcal{R}(X)$ from a topological point of view are studied in
\cite{SHA/GAP, SHA/APROACH, SHA/Continuity}.


\section{Normality}
\begin{proposition} \label{normal1} Suppose $ T \in \mathcal{L}(X)$ admits the polar
decomposition $T= \mathcal{V}|T|$ and $S \in \mathcal{L}(X)$ is
an arbitrary operator which commutes with $T$ and $T^*$. Then $
\mathcal{V}$ and $|T|$ commute with $S$ and $S^*$.
\end{proposition}

\begin{proof} It follows from the hypothesis that
$(T^*T)S=S(T^*T)$ which implies $ |T|S=S|T|$, or equivalently $
|T|S^*=S^*|T|$. Using the commutativity of $S$ with $T$ and
$|T|$, we get
$$(S \mathcal{V} -
  \mathcal{V} S ) |T| = S \mathcal{V} |T| - \mathcal{V} |T| S =ST - TS=0.$$
That is, $ S \mathcal{V} -   \mathcal{V} S $ acts as zero operator
on $ \overline{Ran( |T|)}$. If  $ x \in Ker(|T|)= Ker (
\mathcal{V})$ then $ |T|x= \mathcal{V}x =0$, consequently
$|T|Sx=S|T|x=0$. Then $Sx \in Ker(|T|)= Ker ( \mathcal{V})$,
therefore, $ S \mathcal{V} - \mathcal{V} S $ acts as zero operator
on $Ker(|T|)$ too. We obtain
$$ S \mathcal{V} - \mathcal{V} S =0 \ \ {\rm on } \ \ X=Ker(|T|)
\oplus \overline{Ran(|T|)}. $$  The statement $ S^* \mathcal{V} -
\mathcal{V} S^* =0 \ \ {\rm on } \ \ X=Ker(|T|) \oplus
\overline{Ran(|T|)}$ can be deduced from the commutativity of $S$
with $T^*$ and $|T|$ in the same way.
\end{proof}

\begin{corollary} \label{normal2} Suppose $ T \in \mathcal{L}(X)$ is a normal operator which
admits the polar decomposition $T= \mathcal{V}|T|$ then $
\mathcal{V}$ and $|T|$ commute with the operators $T, T^*,
\mathcal{V}$ and $ \mathcal{V^*}$. In particular, $ \mathcal{V}$
is a unitary operator on $ \overline{Ran(T)} = \overline{Ran(T^*)}
$.
\end{corollary}

The results follow from Proposition \ref{normal1}, Proposition
3.7 of \cite{LAN} and the fact that $
\mathcal{V}\mathcal{V^*}T=\mathcal{V^*}\mathcal{V}T=T$.

\begin{corollary} \label{normal/2} Suppose $ T \in \mathcal{L}(X)$
admits the polar decomposition $T= \mathcal{V}|T|$. Then $T$ is a
normal operator if and only if there exists a unitary operator $
\mathcal{U} \in \mathcal{L}(X)$ commuting with $|T|$ such that $
T= \mathcal{U} |T|$. In this situation, $\mathcal{U}$ also
commutes with $T$ and $T^*$.
\end{corollary}
\begin{proof} Suppose $T$ is a normal operator then
$Ker(T)=Ker(T^*)$ and $\overline{Ran(T)} = \overline{Ran(T^*)}$.
For every $x \in X=Ker(T) \oplus \overline{Ran(T^*)}$ we define

\[ \qquad \mathcal{U}x \, =~
\begin{cases}
~x ~ & \text{ if   $x \in Ker(T)$}\\
\mathcal{V}x ~ & \text{ if $x \in \overline{Ran(T^*)}, $}\
\end{cases}
\] \\
\[ \qquad \mathcal{W}x =
\begin{cases}
~x ~ & \text{ if   $x \in Ker(T^*)$}\\
\mathcal{V}^{*}x ~ & \text{ if $x \in \overline{Ran(T)} $}.\
\end{cases}
\]
Then $ \langle \mathcal{U}x,y \rangle = \langle x, \mathcal{W}y
\rangle $ for all $x,y \in X$, that is,  $ \mathcal{W}=
\mathcal{U}^{*}$. For each $x=x_1 + x_2 \in X$ with $x_1 \in
Ker(T)$ and $x_2 \in \overline{Ran(T^*)}$ we have $$ \mathcal{U}
\mathcal{U^*}x= \mathcal{U}(x_1+ \mathcal{V^*}x_2) = x_1 +
\mathcal{V}\mathcal{V^*}x_2=x.$$ Hence, $ \mathcal{U} \,
\mathcal{U}^{*}=1$ on $X$. We also have $ \mathcal{U}^{*} \,
\mathcal{U}= 1$ and $ T= \mathcal{U} |T|$ on $X$. Commutativity
of $\mathcal{U}$ with $T$, $T^*$ and $|T|$ follows from the
commutativity of $\mathcal{V}$ with $T$, $T^*$ and $|T|$.

Conversely, suppose $ T= \mathcal{U} |T|$ for a unitary operator $
\mathcal{U} \in \mathcal{L}(X)$ which commutes with $|T|$. Then $
T^*= |T| \, \mathcal{U}^* $ and so  $T \, T^*= \mathcal{U} \, |T|
\, |T| \, \mathcal{U}^* =|T| \,  \mathcal{U} \, |T| \,
\mathcal{U}^* =T^* T$.
\end{proof}

\begin{corollary} \label{normal3} Suppose $ T \in \mathcal{L}(X)$
admits the polar decomposition $T= \mathcal{V}|T|$. Then $T$ is a
normal operator if and only if there exists a unitary operator $
\mathcal{U} \in \mathcal{L}(X)$ such that $ T= \mathcal{U} T^*$.
In this situation, $\mathcal{U}$ commutes with $T$ and $T^*$.
\end{corollary}

\begin{proof} Suppose $T$ is a normal operator then $|T|=|T^*|=\mathcal{V} T^*$
and so $T= \mathcal{V}|T|= \mathcal{V}|T^*|=\mathcal{V}^{2} T^*$.
For $x \in X$ we define

\[ \qquad \mathcal{U}x \, =~
\begin{cases}
~x ~ & \text{ if   $x \in Ker(T)$}\\
\mathcal{V}^{2}x ~ & \text{ if $x \in \overline{Ran(T^*)}. $}\
\end{cases}
\]
Then, as in the proof of Corollary  \ref{normal/2},
\[ \qquad \mathcal{U^*}x =
\begin{cases}
~x ~ & \text{ if   $x \in Ker(T^*)$}\\
\mathcal{V}^{* \, 2}x ~ & \text{ if $x \in \overline{Ran(T)} $},\
\end{cases}
\]
which implies $ \mathcal{U}$ is unitary and $ T= \mathcal{U} T^*$.
Commutativity of $\mathcal{U}$ with $T$ and $T^*$ follows from
the commutativity of $\mathcal{V}$ with $T$ and $T^*$.

Conversely, suppose $ T= \mathcal{U} T^*$ for a unitary operator
$ \mathcal{U} \in \mathcal{L}(X)$. Then $ T^*= ( \mathcal{U} T^*
)^* = T \, \mathcal{U}^* $ and so  $T^* T=T \, \mathcal{U}^* \,
\mathcal{U} \, T^*=T \, T^*$.
\end{proof}

If the normal operator $ T \in \mathcal{L}(X)$ has closed range,
one can find a shorter proof for the above result.

\begin{theorem} \label{normal4} Suppose $ t \in \mathcal{R}(X)$
admits the polar decomposition $t= \mathcal{V}|t|$. Then $t$ is a
normal operator if and only if there exists a unitary operator $
\mathcal{U} \in \mathcal{L}(X)$ such that $ t= \mathcal{U} t^*$.
In this situation, $t \, \mathcal{U}=\mathcal{U} \, t $ and $t^*
\, \mathcal{U}=\mathcal{U} \, t^* $ on $Dom(t)=Dom(t^*)$.
\end{theorem}

\begin{proof} Recall that $t$ admits the polar decomposition $t= \mathcal{V}|t|$
if and only if its bounded transform $F_t$ admits the polar
decomposition $F_t= \mathcal{V}|F_t|$, furthermore,  $t$ is a
normal operator if and only if its bounded transform $F_t$ is a
normal operator.

Suppose $t$ is a normal operator then there exists a unitary
operator $ \mathcal{U} \in \mathcal{L}(X)$ such that $t Q_t= F_t=
\mathcal{U} F_{t}^{*}=\mathcal{U} F_{t^*}= \mathcal{U} \, t^*
Q_{t^*}= \mathcal{U} \, t^* Q_{t}$. Since   $Q_{t}:X \rightarrow
Ran(Q_t)=Dom(t)$ is invertible, we obtain $t= \mathcal{U} t^*$.

Conversely, suppose $ t= \mathcal{U} t^*$ for a unitary operator $
\mathcal{U} \in \mathcal{L}(X)$. Then, in view of Remark 2.1 of
\cite{FS2}, we have $ t^*=( \mathcal{U}t^*)^*= t^{**} \,
\mathcal{U}^* = t \, \mathcal{U}^*$ on $Dom(t^*)$ and so  $t^*
t=t \, \mathcal{U}^* \, \mathcal{U} \, t^*=t \, t^*$.

According to Corollary \ref{normal3} and the first paragraph of
the proof, the unitary operator $ \mathcal{U}$ commutes with
$F_t$ and $F_{t}^{*}$. Thus for every polynomial $p$ we have $
\mathcal{U} \, p(F_{t}^{*} F_{t})= p(F_{t}^{*} F_{t}) \,
\mathcal{U}$ and so for every continuous function $p \in {\bf
C}[0,1]$ we have $ \mathcal{U} \, p(F_{t}^{*} F_{t})= p(F_{t}^{*}
F_{t}) \, \mathcal{U}$. In particular, $ \mathcal{U} \, (1-
F_{t}^{*} F_t )^{1/2}= (1 - F_{t}^{*} F_{t})^{1/2} \,
\mathcal{U}$ which implies $ \mathcal{U} \, Q_t= Q_t \,
\mathcal{U}$. This fact together with the equality $F_t \,
\mathcal{U}= \mathcal{U} F_t$ imply that $ t \, \mathcal{U} \, Q_t
=  t Q_t \, \mathcal{U} = \mathcal{U} \, t Q_t $. Again by
invertibility of the map $Q_{t}:X \rightarrow Ran(Q_t)=Dom(t)$ we
obtain $t \, \mathcal{U}= \mathcal{U} \, t $ on $Dom(t)$. To
demonstrate the second equality we have $ \mathcal{U^*} \,
t=\mathcal{U^*} \, \mathcal{U} \, t^*=t^*$ which yields  $ t^* \,
\mathcal{U} = ( \mathcal{U^*} \, t)^* = t^{**}=t= \mathcal{U} \,
t^* $ on $Dom(t^*)$.
\end{proof}

\begin{corollary} \label{normal5} Suppose $ t \in \mathcal{R}(X)$
has closed range. Then $t$ is a normal operator if and only if
there exists a unitary operator $ \mathcal{U} \in \mathcal{L}(X)$
such that $ t= \mathcal{U} t^*$. In this situation, $t \,
\mathcal{U}=\mathcal{U} \, t $ and $t^* \,
\mathcal{U}=\mathcal{U} \, t^* $ on $Dom(t)=Dom(t^*)$.
\end{corollary}

The proof immediately follows from Theorem \ref{normal4},
Proposition 1.2 of \cite{F-S} and Theorem 3.1 of \cite{FS2}.

\begin{corollary} \label{normal5-2} Suppose $X$ is a Hilbert space
(or a Hilbert C*-module over an arbitrary C*-algebra of compact
operators) and $t: Dom(t)  \subseteq X \to X$ is a densely
defined closed operator. Then $t$ is a normal operator if and
only if there exists a unitary operator $ \mathcal{U} \in
\mathcal{L}(X)$ such that $ t= \mathcal{U} t^*$. In this
situation, $t \, \mathcal{U}=\mathcal{U} \, t $ on
$Dom(t)=Dom(t^*)$.
\end{corollary}


Consider two normal operators $T$ and $S$ on a Hilbert space it is
known that, in general, $TS$ is not normal. Historical notes and
several versions of the problem are investigated in
\cite{Gheondea/normal}. Kaplansky has shown that it may be
possible that $TS$ is normal while $ST$ is not. Indeed, he has
shown that if $T$ and $TS$ are normal, then $ST$ is normal if and
only if $S$ commutes with $|T|$, cf.
\cite{Kaplansky/normality1953}. We generalize his result for
bounded adjointable operator on Hilbert C*-modules. For this aim
we need the Fuglede-Putnam theorem for bounded adjointable
operators on Hilbert C*-modules. Using Theorem 4.1.4.1 of
\cite{Constantinescu} for the unital C*-algebra $ \mathcal{L}(X)$,
we obtain:

\begin{theorem}\label{Fuglede-Putnam}{\bf (Fuglede-Putnam)} Assume that $T,
S$ and $A$ are bounded adjointable operator in $ \mathcal{L}(X)$.
Suppose $T$ and $S$ are normal and $T A = A S$, then $T^* A= A
S^*$.
\end{theorem}

\begin{corollary}Let $T, S \in
\mathcal{L}(X)$ be such that $T$ and $TS$ are normal and $T$ has
polar decomposition. $ST$ is normal if and only if $S$ commutes
with $|T|$.
\end{corollary}

\begin{proof} Suppose $ST$ and $T$ are normal operators and $A=TS$ and $B=ST$,
then $AT=TB$. In view of the Theorem \ref{Fuglede-Putnam},
$A^*T=TB^*$, that is, $S^*T^*T=T \, T^* S^*$, and taking into
account the normality of $T$, we find $S^*$ commutes with $T^*T$.
Therefore, $S^* |T|=|T| S^*$ and so $S$ commutes with $|T|$ by
the Fuglede-Putnam theorem.

Conversely, suppose $S$ commutes with $|T|$. Then the normal
operator $T$ has a representation $T= \mathcal{U}\, |T|$ in which
$ \mathcal{U} \in \mathcal{L}(X)$ is unitary and commutes with
$|T|$. Therefore,
$$ \mathcal{U}^* \, TS \,  \mathcal{U}=   \mathcal{U}^* \,
\mathcal{U}\,  |T| \, S  \, \mathcal{U} = S |T| \, \mathcal{U} = S
\, \mathcal{U}  \, |T|=ST .$$ The operator $ST$ is normal as an
operator which is unitary equivalent with the normal operator
$TS$.
\end{proof}


\end{document}